\documentclass[12pt]{amsart}
\newtheorem{thm}{Theorem}
\newtheorem{lemma}{Lemma}
\begin{document}
\title[Growth and repair]
{Volume growth and puncture repair in conformal geometry}
\author[M.G. Eastwood]{Michael G. Eastwood}
\address{\hskip-\parindent
School of Mathematical Sciences\\
University of Adelaide\\ 
SA 5005\\ 
Australia}
\email{meastwoo@member.ams.org}
\author[A.R. Gover]{A. Rod Gover}
\address{Department of Mathematics\\
University of Auckland\\
Private Bag 92019\\ 
Auckland 1142\\
New Zealand}
\email{r.gover@auckland.ac.nz}
\begin{abstract} Suppose $M$ is a compact Riemannian manifold and $p\in M$ an
arbitrary point. We employ estimates on the volume growth around $p$ to prove
that the only conformal compactification of $M\setminus\{p\}$ is $M$ itself.
\end{abstract}
\renewcommand{\subjclassname}{\textup{2010} Mathematics Subject Classification}

\thanks{We gratefully acknowledge support from the Royal Society of
New Zealand via Marsden Grants 13-UOA-018 and 16-UOA-051. We also thank the
Universities of Adelaide and Auckland for hospitality during various visits.}

\maketitle
\section{Introduction}
Though this article is primarily concerned with conformal differential geometry
in dimension $\geq3$, the phenomenon we wish to describe also occurs in
dimension~$2$ as follows. 
\begin{thm}\label{1} Suppose that $M$ is a compact connected Riemann surface
and $p\in M$. Suppose that $N$ is a compact connected Riemann surface and
$U\subset N$ an open subset such that $U\cong M\setminus\{p\}$ as Riemann
surfaces. Then this isomorphism extends to $N\cong M$.
\end{thm}
Stated more informally, there is no difference between the `punctured Riemann
surface' $M\setminus\{p\}$ and the `marked Riemann surface' $(M,p)$. In fact,
although we have stated the theorem in terms of {\em compact\/} Riemann
surfaces, the result itself is local:
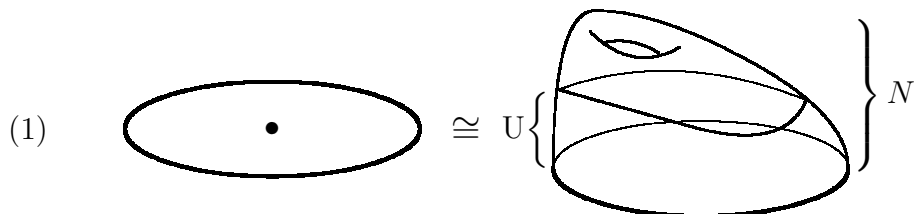
\begin{equation}\label{picture}\raisebox{-27pt}{\setlength{\unitlength}{.7cm}
\begin{picture}(6,3)(.1,0.45)
\linethickness{0.4mm}
\qbezier(5.8,2.0)(5.8,2.3728)(4.9799,2.6364)
\qbezier(4.9799,2.6364)(4.1598,2.9)(3.0,2.9)
\qbezier(3.0,2.9)(1.8402,2.9)(1.0201,2.6364)
\qbezier(1.0201,2.6364)(0.2,2.3728)(0.2,2.0)
\qbezier(0.2,2.0)(0.2,1.6272)(1.0201,1.3636)
\qbezier(1.0201,1.3636)(1.8402,1.1)(3.0,1.1)
\qbezier(3.0,1.1)(4.1598,1.1)(4.9799,1.3636)
\qbezier(4.9799,1.3636)(5.8,1.6272)(5.8,2.0)
\put(3,2){\makebox(0,0){$\bullet$}}
\end{picture}\enskip\raisebox{26pt}{\large$\cong$}\qquad
\begin{picture}(6,4)(0,1.2)
\linethickness{0.1mm}
\qbezier(5.8,2.0)(5.8,2.3728)(4.9799,2.6364)
\qbezier(4.9799,2.6364)(4.1598,2.9)(3.0,2.9)
\qbezier(3.0,2.9)(1.8402,2.9)(1.0201,2.6364)
\qbezier(1.0201,2.6364)(0.2,2.3728)(0.2,2.0)
\linethickness{0.4mm}
\qbezier(0.2,2.0)(0.2,1.6272)(1.0201,1.3636)
\qbezier(1.0201,1.3636)(1.8402,1.1)(3.0,1.1)
\qbezier(3.0,1.1)(4.1598,1.1)(4.9799,1.3636)
\qbezier(4.9799,1.3636)(5.8,1.6272)(5.8,2.0)
\linethickness{0.25mm}
\qbezier(5.8,2)(5.8,3)(4,4)
\qbezier(4,4)(2.2,5)(1,5)
\qbezier(1,5)(.2,5)(.2,2)
\qbezier(1.2,4.4)(1.7,4.5)(2.2,4.2)
\qbezier(.9,4.6)(1.7,3.8)(2.6,4.3)
\qbezier(.3,3.5)(1,3.3)(2.8,2.8)
\qbezier(2.8,2.8)(4.6,2.3)(5,3.3)
\linethickness{0.1mm}
\qbezier(.3,3.5)(2.5,4.3)(5,3.3)
\put(-.3,2.75){\makebox(0,0){U$\left\{\rule{0pt}{18pt}\right.$}}
\put(6.45,3.45){\makebox(0,0){$\left.\rule{0pt}{34pt}\right\}N$}}
\end{picture}}\end{equation}
In this picture the punctured open disc is assumed to be conformally isomorphic
to the open set $U$ (but nothing is supposed concerning the boundary
$\partial{U}$ of $U$ in~$N$). We may conclude that $N$ must be, in fact, be the
disc and $U\hookrightarrow N$ the punctured disc, tautologically included. For
simplicity, however, the results in this article will be formulated for compact
manifolds, their local counterparts being left to the reader.

By a {\em conformal\/} manifold we shall mean a smooth manifold equipped 
with an equivalence class of Riemannian metrics $[g_{ab}]$ where the notion of 
equivalence is that $\hat{g}_{ab}=\Omega^2g_{ab}$ for some positive smooth 
function~$\Omega$.  

\begin{thm}\label{2} Suppose $M$ is a compact connected conformal manifold and
$p\in M$. Suppose $N$ is a compact connected conformal manifold and $U\subset
N$ an open subset such that $U\cong M\setminus\{p\}$ as conformal manifolds.
Then this isomorphism extends to $N\cong M$.
\end{thm}

Since an oriented conformal structure in $2$ dimensions is the same as a
complex structure, Theorem~\ref{2} generalises Theorem~\ref{1}. It is well
known, however, that conformal geometry in dimensions $\geq3$ enjoys a greater
rigidity than in $2$ dimensions and so one expects a different proof. Such
proofs of Theorem~\ref{2} (and beyond) can be found in~\cite{F}. 
In this article, however, we shall prove Theorem~\ref{2} by a method that 
also works (but much more easily so) in dimension~$2$. 

\section{Puncture repair in $2$ dimensions}
\noindent{\em Proof of Theorem~\ref{1}.}\enskip With reference to
picture~(\ref{picture}), introducing polar co\"ordinates $(r,\theta)$ on the
disc and hence on~$U$, we are confronted by a smooth positive function
$\Omega(r,\theta)$ so that, if the $\eta_{ab}$ denotes the standard metric
$dr^2+r^2d\theta^2$ on the disc, then the metric
$\hat{\eta}_{ab}=\Omega^2\eta_{ab}$ extends to~$N$. If $\partial U\subset N$
contains two or more points, then the concentric curves $\{r=\epsilon\}$ as
$\epsilon\downarrow 0$ have length bounded away from zero in the metric
$\hat{\eta}_{ab}$. In other words, for some~$\epsilon>0$ and $\ell>0$, we have
$$\int_0^{2\pi}\Omega(r,\theta)\,r\,d\theta\geq\ell,
\qquad\forall\,0<r<\epsilon.$$
By the Cauchy-Schwarz inequality, for any fixed~$r$,
$$\left(\int_0^{2\pi}\Omega(r,\theta)\,d\theta\right)^2\leq
2\pi\int_0^{2\pi}\Omega^2(r,\theta)\,d\theta$$
and it follows that
$$\int_0^\epsilon\!\int_0^{2\pi}\Omega^2\,d\theta\,r\,dr\geq
\int_0^\epsilon\frac1{2\pi}
\left(\int_0^{2\pi}\Omega\,d\theta\right)^2\!r\,dr\geq
\frac1{2\pi}\int_0^\epsilon\frac{\ell^2}{r}\,dr=\infty.$$
However, the integral on the left is the area of $\{0<r<\epsilon\}\subseteq U$
in the metric $\hat\eta_{ab}$, which must be finite if $\hat\eta_{ab}$
is to extend smoothly to~$N$. \hfill$\square$

\section{Puncture repair in Euclidean $n$-space}
In $2$ dimensions, the local existence of isothermal co\"ordinates implies that
is it sufficient to repair only the unit disc in ${\mathbb{R}}^2$ with its
standard metric~$\eta_{ab}$. Such a normalisation is unavailable in higher
dimensions.

\noindent{\em Proof of Theorem~\ref{2} in flat space.}\enskip With reference
to~(\ref{picture}), now viewed as a picture in $n$ dimensions, we shall suppose
that the object on the left is a punctured ball in ${\mathbb{R}}^n$ with its
standard Euclidean metric and aim to conclude, just as we did in case $n=2$,
that $\partial U\subset N$ is a single point. To do this, we replace polar
co\"ordinates by spherical co\"ordinates
$${\mathbb{R}}_{>0}\times\Sigma\ni(r,x)\mapsto
rx\in{\mathbb{R}}^n\setminus\{0\},$$
where
$$\Sigma=\{x\in{\mathbb{R}}^n\mid\|x\|=1\}
\hookrightarrow{\mathbb{R}}^n\setminus\{0\}$$
is the unit $(n-1)$-sphere and investigate the behaviour of a smooth positive
function $\Omega=\Omega(r,x)$ defined for $r$ sufficiently small and having the
property that the metric $\hat\eta_{ab}=\Omega^2\eta_{ab}$ extends to~$N$. If 
$\partial U\subset N$ contains two or more points, then the concentric 
hypersurfaces $\{r=\epsilon\}$ as $\epsilon\downarrow0$ have diameter bounded 
away from zero in the metric~$\hat\eta_{ab}$. In other words, for some 
$\epsilon>0$ and $\ell>0$, we have
$$\forall\,0<r<\epsilon,\quad\mbox{there are $\alpha,\beta\in\Sigma$}\quad
\mbox{s.t.\ }\int_\alpha^\beta\Omega(r,x)\,r \geq\ell,$$
where the integral is along any path from $\alpha$ to $\beta$ on the unit
sphere~$\Sigma$ (with respect to the standard round metric on~$\Sigma$).

\begin{lemma}\label{keylemma} Suppose $\Omega:\Sigma\to{\mathbb{R}}_{>0}$ is
smooth and there are two points $\alpha,\beta\in\Sigma$ such that
$\int_\alpha^\beta\Omega\geq d>0$ for all smooth paths on $\Sigma$ joining
$\alpha$ to~$\beta$. Then
$$\int_\Sigma\Omega^n\geq C_n\,d^n,$$
where $C_n$ is a universal constant, independent of the location 
of~$\alpha,\beta$.
\end{lemma}

\noindent The proof of this lemma is given in an appendix. To finish the proof 
of our theorem, we compute the volume of the collar $\{0<r<\epsilon\}$ with 
respect to the metric $\hat\eta_{ab}=\Omega^2\eta_{ab}$ as
$$\int_0^\epsilon\!\int_\Sigma\Omega^n\,r^{n-1}\,dr\geq
\int_0^\epsilon C_n\left(\frac{\ell}{r}\right)^nr^{n-1}\,dr
=C_n\int_0^\epsilon\frac{\ell^n}{r}\,dr=\infty,$$
which should be finite if $\hat\eta_{ab}$ is to extend smoothly to~$N$.
\hfill$\square$

\medskip\noindent{\em Remark.} By stereographic projection, conformally
repairing a puncture in Euclidean ${\mathbb{R}}^n$ is equivalent to conformally
repairing a puncture in the round sphere~$S^n$. It follows already that $S^n$ 
is the unique conformal compactification of~${\mathbb{R}}^n$. 

\section{Puncture repair near the Euclidean metric}
The estimates in the previous section are sufficiently robust that they apply
for metrics sufficiently close to Euclidean. More specifically, suppose
$g_{ab}(x)\,dx^adx^b$ is a Riemannian metric on a punctured ball in
${\mathbb{R}}^n$ centred on the origin and such that, in standard Cartesian
co\"ordinates $(x^1,x^2,\ldots,x^n)\in{\mathbb{R}}^n$ with standard Euclidean
metric $\eta_{ab}\,dx^adx^b$,
\begin{equation}\label{close_to_Euclid}
\begin{tabular}{l}
$\bullet$\enskip the volume form for $g_{ab}$ is the standard Euclidean one,\\
$\bullet$\enskip the metrics $g_{ab}(x)$ and $\eta_{ab}$ satisfy\\[4pt]
\qquad\qquad
$\|X\|_{g(x)}\leq2\|X\|_\eta\leq 
4\|X\|_{g(x)},\quad\forall\,\mbox{vectors }X$\\[4pt]
for all $x$ near the origin, say for $\|x\|_\eta<\epsilon$.
\end{tabular}
\end{equation}

Again working in spherical coordinates near the origin, if
$\hat{g}_{ab}\equiv\Omega^2g_{ab}$ on $U$ extends to $N$, then the
concentric hypersurfaces $\{r=\epsilon\}$ as $\epsilon\downarrow 0$ have
diameter bounded away from zero in the metric $\hat{g}_{ab}$ and hence also in
the commensurate metric $\hat{\eta}_{ab}\equiv\Omega^2\eta_{ab}$. Therefore,
according to the proof given in the previous section, the volume of the collar 
$\{0<r<\epsilon\}$ with respect to the metric $\hat\eta_{ab}$ is infinite. But 
$\hat\eta_{ab}$ has the same volume form as $\hat{g}_{ab}$, which contradicts 
$\hat{g}_{ab}$ extending to~$N$.

\section{Puncture repair in $n$ dimensions}
\noindent{\em Proof of Theorem~\ref{2} in general.}\enskip We only need show
that there are local co\"ordinates on a arbitrary Riemannian manifold so that
conditions (\ref{close_to_Euclid}) are satisfied. Certainly, we can arrange
local co\"ordinates so that $g_{ab}$ agrees with $\eta_{ab}$ at the origin. 
The volume form for $g_{ab}$ is then
$$F(x^1,x^2\cdots,x^n)\,dx^1\wedge dx^2\wedge\cdots\wedge dx^n$$
for some smooth function $F$ with~$F(0)=1$, which can be absorbed by changing
just the first co\"ordinate. The condition that $g_{ab}$ and $\eta_{ab}$ are
commensurate, as in~(\ref{close_to_Euclid}), follows near the origin by
continuity. \hfill$\square$

\appendix
\section{Proof of Lemma~\ref{keylemma}}
In fact, we shall prove the following minor generalisation. 
\begin{lemma} Let $\Sigma$ denote the unit $m$-sphere with its usual round 
metric. Suppose $\Omega:\Sigma\to{\mathbb{R}}_{>0}$ is
smooth and there are two points $\alpha,\beta\in\Sigma$ such that
$\int_\alpha^\beta\Omega\geq d>0$ for all smooth paths on $\Sigma$ joining
$\alpha$ to~$\beta$. Then, for any $s>m$, 
$$\int_\Sigma\Omega^s\geq C_{m,s}\,d^s,$$
where $C_{m,s}$ is a universal positive constant, independent of the location 
of $\alpha$ and~$\beta$.
\end{lemma}
\begin{proof}
We shall calculate using stereographic co\"ordinates on~$\Sigma$. Recall that 
the round metric on the unit $m$-sphere may be written as
$$\left(\frac{2}{1+(u^1)^2+\cdots+(u^m)^2}\right)^2
\Big((du^1)^2+\cdots+(du^m)^2\Big)$$
in these co\"ordinates. Translating the origin to $a\in{\mathbb{R}}^m$
gives
$$\left(\frac{2}{1+\|u+a\|^2}\right)^2
\Big((du^1)^2+\cdots+(du^m)^2\Big)$$
instead and we may use such a translation to suppose that $\alpha$ and $\beta$ 
are located at 
the origin and out at infinity in this stereographic projection. Let 
$v\in{\mathbb{R}}^m$ be a unit vector and consider the curve
$$(0,\infty)\ni \rho\mapsto\rho v$$
joining the origin to infinity. Then 
$$d\leq\int_\alpha^\beta\Omega
=\int_0^\infty \frac{2\Omega\,d\rho}{1+\|\rho v+a\|^2}.$$
The H\"older inequality for conjugate exponents $s$ and $s/(s-1)$ implies that
$$\left(\int_0^\infty\!\!fg\,d\mu\right)^s\leq
\left(\int_0^\infty\!\!f^s\,d\mu\right)
\left(\int_0^\infty\!\!g^{s/(s-1)}\,d\mu\right)^{s-1}$$
and if we take 
$$f=\Omega,\quad g=\frac{(1+\|\rho v+a\|^2)^{m-1}}{\rho^{m-1}},\quad
d\mu=\frac{\rho^{m-1}\,d\rho}{(1+\|\rho v+a\|^2)^m}$$
then we conclude that
$$\left(\int_0^\infty\frac{\Omega\,d\rho}{1+\|\rho v+a\|^2}\right)^s\leq
A_{v}{}^{s-1}
\int_0^\infty\frac{\Omega^s\,\rho^{m-1}\,d\rho}{(1+\|\rho v+a\|^2)^m},$$ 
where
$$A_v\equiv\int_0^\infty
\frac{d\rho}{\rho^{(m-1)/(s-1)}(1+\|\rho v+a\|^2)^{(s-m)/(s-1)}}.$$
But 
$$\|\rho v+a\|^2=(\rho+\langle a,v\rangle)^2+\|a-\langle a,v\rangle v\|^2\geq
(\rho+\langle a,v\rangle)^2$$
so we conclude that
$$A_v\leq
\int_0^\infty
\frac{d\rho}{\rho^{(m-1)/(s-1)}(1+(\rho+\langle a,v\rangle)^2)^{(s-m)/(s-1)}}$$
an integral that is bounded,
say by $B$, independent of 
$\langle a,v\rangle\in{\mathbb{R}}$, by dint of Lemma~\ref{tricky_lemma} below.
We conclude that
$$\int_0^\infty\frac{\Omega^s\,\rho^{m-1}\,d\rho}{(1+\|\rho v+a\|^2)^m}\geq
\frac1{B^{s-1}}
\left(\int_0^\infty\frac{\Omega\,d\rho}{1+\|\rho v+a\|^2}\right)^s.$$
Recall that $v$ is an arbitrarily chosen unit vector in ${\mathbb{R}}^m$. 
If we integrate over all such vectors, then the left hand side of this 
inequality yields $\int_\Sigma\Omega^s$ whereas we already know that 
$$\int_0^\infty\frac{\Omega\,d\rho}{1+\|\rho v+a\|^2}\geq\frac{d}2.$$
Therefore $\int_\Sigma\Omega^s\geq(S_{m-1}/(2^sB^{s-1}))\,d^s$, where 
$S_{m-1}$ denotes the area of the unit $(m-1)$-sphere. This is a bound of the
required form.
\end{proof}

\begin{lemma}\label{tricky_lemma}
If $p,q>1$ are conjugate exponents, $1/p+1/q=1$, then 
$${\mathbb{R}}\ni
t\longmapsto\int_0^\infty\frac{dx}{x^{1/p}(1+(x+t)^2)^{1/q}}$$
is bounded.\end{lemma}
\begin{proof}
Firstly, note that $(q+1)/2q>1/2$, so
$$\int_1^\infty\frac{dx}{(1+(x+t)^2)^{(q+1)/2q}}\leq
\int_{-\infty}^\infty\frac{dx}{(1+x^2)^{(q+1)/2q}}<\infty.$$
Since $p$ and $q$ are conjugate exponents, so are
$2p-1$ and $(q+1)/2$, and we may apply Young's inequality with these exponents 
to conclude that 
$$\frac{1}{x^{1/p}(1+(x+t)^2)^{1/q}}\leq
\frac{1}{2p-1}\frac{1}{x^{(2p-1)/p}}
+\frac{2}{q+1}\frac{1}{(1+(x+t)^2)^{(q+1)/2q}}.$$
Also, observe that 
$$\int_0^1\frac{dx}{x^{1/p}}
=\int_1^\infty\frac{dx}{x^{(2p-1)/p}}
=\frac{p}{p-1}.$$
Therefore,
$$\begin{array}{rcl}\displaystyle
\int_0^\infty\frac{dx}{x^{1/p}(1+(x+t)^2)^{1/q}}
&\leq&\displaystyle\int_0^1\frac{dx}{x^{1/p}}
+\int_1^\infty\frac{dx}{x^{1/p}(1+(x+t)^2)^{1/q}}\\[12pt]
&=&\displaystyle\frac{p}{p-1}
+\int_1^\infty\frac{dx}{x^{1/p}(1+(x+t)^2)^{1/q}}
\end{array}$$
whilst Young's inequality shows that the second integral is bounded above by
$$\frac{1}{2p-1}\int_1^\infty
\frac{dx}{x^{(2p-1)/p}}
+\frac{2}{q+1}\int_1^\infty\frac{dx}{(1+(x+t)^2)^{(q+1)/2q}}.$$
Assembling these various estimates gives
$$\frac{2p^2}{(2p-1)(p-1)}
+\frac{2}{q+1}\int_{-\infty}^\infty\frac{dx}{(1+x^2)^{(q+1)/2q}}$$
as a bound on the original integral.\end{proof}
\subsection*{Acknowledgement} We are extremely grateful to Ben Moore and Nick 
Buchdahl for helpful discussions concerning the proof of 
Lemma~\ref{tricky_lemma}.


\begin{thebibliography}{11}

\bibitem{F} C. Frances, 
{\em Removable and essential singular sets for higher dimensional conformal 
maps}, 
Comment. Math. Helv. {\bf 89} (2014) 405--441.

\end{thebibliography}
\end{document}